\documentclass[12pt]{article} \usepackage{fullpage}
\usepackage{lineno}
\usepackage{amsmath,hyperref}
\usepackage{cleveref}
\usepackage{nicefrac}
\usepackage{url}
\usepackage{float}
\usepackage{color}
\usepackage[american]{babel}
\usepackage[inline]{enumitem}
\usepackage{placeins}
\usepackage{graphicx}
\usepackage{sansmath}
\usepackage{version}
\usepackage{subcaption}
\usepackage[utf8]{inputenc}
\usepackage{polski}
\usepackage[sort&compress,numbers]{natbib}
\usepackage[bold,full]{complexity}
\usepackage[textsize=tiny]{todonotes}
\def\changed#1{{\color{black}#1}}
\usepackage{tablefootnote}
\newif\iffull
\fulltrue

\graphicspath{{figures/}}

\usepackage{amsthm,amsmath,amssymb}
\newtheorem{theorem}{Theorem}

\newtheorem{corollary}{Corollary}

\newtheorem{problem}{Problem}
\newtheorem{claim}{Claim}

\newtheorem{lemma}{Lemma}
\usepackage[symbol]{footmisc}
\newcommand{\student}[1]{}
\newcommand{\postdoc}[1]{}

\renewcommand{\subsubsection}[1]{\paragraph{#1}}

\newcommand{\leaveout}[1]{}

\makeatletter
\renewcommand{\@fnsymbol}[1]{%
  \ifcase#1
  \or §
  \or ¶
  \or *
  \else \@ctrerr
  \fi
}
\makeatother

\date{}
\title{On the $d$-independence number in 1-planar graphs} 
\author{Therese Biedl%
\thanks{David R.~Cheriton School of Computer
Science, University of Waterloo, Waterloo, Ontario N2L 3G1, Canada.
Work of TB supported by NSERC\changed{; FRN RGPIN-2020-03958}.  {\tt biedl@uwaterloo.ca} 
}
\and Prosenjit Bose\thanks{School of Computer Science,  Carleton University, Ottawa, Canada. Research supported in part by NSERC. {\tt  jit@scs.carleton.ca}}
\and Babak Miraftab\thanks{School of Computer Science,  Carleton University, {\tt  babakmiraftab@cunet.carleton.ca}}
}

\begin{document}

\maketitle
\begin{abstract}
The \emph{$d$-independence number} of a graph $G$ is the largest possible size of an independent set $I$ in $G$
where each vertex of $I$ has degree at least $d$ in $G$.    Upper bounds for the $d$-independence number in planar
graphs are well-known for $d=3,4,5$, and can in fact be matched with constructions that actually have minimum degree $d$.   In
this paper, we explore the same questions for 1-planar graphs, i.e., graphs that can be drawn in the plane with at most
one crossing per edge.   We give upper bounds for the $d$-independence number for all $d$. Then we give
constructions that match the upper bound, and (for small $d$) also have minimum degree $d$.

\end{abstract}

\section{Introduction}

An \emph{independent set} in a graph contains vertices that are not adjacent to each other.
A \emph{maximum independent set} is an independent set of largest possible size for a given graph, and its cardinality is known as the \emph{independence number} of $G$ and denoted by $\alpha(G)$.
The celebrated 4-color theorem~\cite{appel89, Rob97} immediately implies that $\alpha(G)\geq n/4$ for every planar graph $G$, where $n$ is the number of vertices in the graph. Interestingly, this bound represents the maximum attainable, as there exist planar graphs without larger independent sets; for instance, consider disjoint copies of complete graphs with $4$ vertices. 
Some weaker lower bounds are also established \cite{ALBERTSON197684,CranstonR16a} that circumvent the complexity of the 4-color theorem (as suggested by Erd\H{o}s \cite{Berge1970}) via charging and discharging argument.

From an algorithmic standpoint, determining the maximum independent set in planar graphs is \NP-hard, even when restricted to planar graphs of maximum degree 3, see \cite{GareyJS76,Mohar01} or planar triangle-free graphs, see \cite{Madhavan84}. 
Consequently, efforts have shifted towards finding large independent sets through methods like approximation algorithms \cite{AlekseevLMM08,Baker94,Burns89, ChibaNS83,DvorakM17, Madhavan84}, parallel algorithms \cite{ChrobakN91,DadounK90,He88}, 
fixed-parameter tractable algorithms
\cite{fomin2024hamiltonicity,lokshantov2014independent,MR4568165,MR2922071}
or within certain minor-free planar graphs \cite{DvorakM17,LozinM10,MagenM09}.

The size $\alpha(G)$ serves as a crucial parameter in graph theory and holds significance in algorithmic contexts. 
For instance, Kirkpatrick~\cite{Kirk83} and Dobkin and Kirkpatrick~\cite{Dob85} employed the repeated removal of independent sets from triangulations to devise data structures for efficient point location and distance computation between convex polytopes, respectively.     For this technique, it is important that the vertices in the independent set have bounded degrees.
Biedl and Wilkinson~\cite{BW05} further explored the size of bounded degree independent sets in triangulations.
In addition, Bose,  Dujmović and Wood \cite{MR2213869} studied how to obtain graphs with large independent sets of bounded degrees in graphs of bounded treewidth.

In the above context, `bounded degree independent set' meant that every vertex of the independent set should have degree \emph{at most} $d$, for some specified constant $d$.    In contrast to this, we study here independent sets where every vertex of the independent set should have degree \emph{at least} $d$.    We call this a \emph{$d$-independent set}%
\footnote{The term `$d$-independent set' is heavily overloaded in the literature; for example, it has also been used for an independent set of cardinality $d$ \cite{TwinwidthIII} or for an independent set that induces a graph of maximum degree $d$ \cite{kIndSurvey}.   Unfortunately, more descriptive notations such as `$(\deg \geq d)$-independent set' are rather difficult to parse.
}
and the maximum size of a $d$-independent set in $G$ is called the \emph{$d$-independence number} of $G$ and denoted $\alpha_d(G)$.    We also speak of the independence number of a graph class $\mathcal{G}$ (such as `the 1-planar graphs'), which is a function that depends on $n$ and denotes the maximum of $\alpha_d(G)$ over all graphs $G$ with $n$ vertices in~$\mathcal{G}$.

The $d$-independence number for planar graphs was first studied by 
Caro and Roditty in \cite{MR0824858} (they actually studied graphs of minimum
degree $d$, which means that any independent set is a $d$-independent set).
They showed that any independent set in a simple planar graph $G$ with minimum
degree $d$ has size at most $\tfrac{2n-4}{d}$.
%
They also showed that these bounds are tight for $2\leq d\leq 5$
by constructing infinite families of simpler planar graphs with minimum degree $d$
and an independent set of size $\tfrac{2n-4}{d}$.   

There are various generalizations of planar graphs,  for example, a \emph{$1$-planar} graph is a graph that can be drawn in the Euclidean plane with at most one crossing per edge.  In this paper, we study the independence number of $1$-planar graphs.
Borodin \cite{MR0832128} establishes that every $1$-planar graph $G$ has chromatic number at most $6$, therefore $\alpha(G)\geq n/6$. This is tight; for example, a graph consisting of disjoint copies of $K_6$ is $1$-planar and has chromatic number $6$.
Unlike planar graphs, to the best of our knowledge, there are no prior results\footnote{Preliminary version appeared in \cite{biedl2024independence}.}
on the upper bound for the $d$-independence number of $1$-planar graphs.
Motivated by the work by Caro and Roditty, we ask the following:
\begin{problem}\label{pro_1}
    What is the $d$-independence number of 1-planar graphs?   And does it change if we additionally require the graphs to have minimum degree $d$? 
\end{problem}

In this paper, we address Problem 1 by 
giving upper bounds for the $d$-independence number of 1-planar graphs for $d\geq 3$.
(No smaller values of $d$ are worth studying, since $K_{2,n}$ has a 2-independent set
of size $n-2$, and so only trivial bounds of $n-O(1)$ can be shown.)
Specifically, we show that $\alpha_d(n)\leq \tfrac{4}{d+\lceil d/3 \rceil}(n-2)$ for $d\geq 3$,
and this bound holds not only for simple graphs but also in the presence of parallel
edges as long as (in some 1-planar drawing) there are no bigons.    For simple graphs
we can prove the stronger bound $\alpha_3(n)\leq \tfrac{6}{7}(n-2)$.   Then we construct
classes of 1-planar graphs that match these bounds (at least if we permit parallel edges);
for small values of $d$, we can construct the graphs such that they are simple and have minimum
degree $d$. 
A concise summary of our results is in Table~\ref{ta:overviewNew}.
\def\myapprox{\ensuremath{{\scriptstyle{\approx}}}}

\begin{table}[ht] \centering
\renewcommand{\arraystretch}{1.3}
\begin{tabular}{|@{}c|%
@{\,}c@{\,}|%
@{\,}c@{\,}|@{\,}c@{\,}|@{\,}c@{\,}|@{\,}c@{\,}|@{\,}c@{\,}|@{\,}c@{\,}|@{\,}c@{\,}|@{\,}c@{\,}|@{\,}c@{\,}|@{\,}c@{\,}|@{\,}c@{\,}|@{\,}c@{\,}|@{\,}c@{\,}|@{\,}c@{\,}|@{\,}c@{\,}|@{\,}c@{}|}
\hline
$d$ & $3$ & $4$ & $5$ & $6$ & $7$ & $8$ & $9$ & $10$ & $11$ & $12$ & $13$ & $14$ & $15$ & $16$ & $17$ & $18$ & ${\geq}19$ \\
\hline
\hline
$\alpha_d(G)/n \leq$ & 
$\myapprox \tfrac{6}{7}$ 
& $\myapprox \tfrac{2}{3}$ 
& $\myapprox \tfrac{4}{7}$ 
& $\myapprox \tfrac{1}{2}$ 
& $\myapprox \tfrac{2}{5}$ 
& $\myapprox \tfrac{4}{11}$ 
& $\myapprox \tfrac{1}{3}$ 
& $\myapprox \tfrac{2}{7}$ 
& $\myapprox \tfrac{4}{15}$ 
& $\myapprox \tfrac{1}{4}$ 
& $\myapprox \tfrac{2}{9}$ 
& $\myapprox \tfrac{4}{19}$ 
& $\myapprox \tfrac{1}{5}$ 
& $\myapprox \tfrac{2}{11}$ 
& $\myapprox \tfrac{4}{23}$ 
& $\myapprox \tfrac{1}{6}$ 
& $\myapprox \tfrac{4}{d+\lceil d/3 \rceil}$ 
\\[-1.2ex]
{\scriptsize ($G$ is bigon-free)} & {\tiny ($G$ simple)} & & & & & & & & & & & & & & & & \\
\hline
$\,\exists G: \alpha_d(G)/n \geq$ & 
$\myapprox \tfrac{6}{7}$ 
& $\myapprox \tfrac{2}{3}$ 
& $\myapprox \tfrac{4}{7}$ 
& $\myapprox \tfrac{1}{2}$ 
& $\myapprox \tfrac{2}{5}$ 
& $\myapprox \tfrac{1}{3}$ 
& $\myapprox \tfrac{2}{7}$ 
& $\myapprox \tfrac{1}{4}$ 
& $\myapprox \tfrac{3}{13}$ 
& $\myapprox \tfrac{3}{14}$ 
& $\myapprox \tfrac{1}{5}$ 
& $\myapprox \tfrac{2}{11}$ 
& $\myapprox \tfrac{1}{6}$ 
& $\myapprox \tfrac{3}{19}$ 
& $\myapprox \tfrac{3}{20}$ 
& $\myapprox \tfrac{1}{7}$ 
& $\myapprox\tfrac{2}{d-4}$ \\[-1.2ex]
{\scriptsize ($G$ is simple)} & & & & & & & & & & & & & & & & & \\
\hline
$\,\exists G: \alpha_d(G)/n \geq$ & 
$\myapprox 1$ 
& $\myapprox \tfrac{2}{3}$ 
& $\myapprox \tfrac{4}{7}$ 
& $\myapprox \tfrac{1}{2}$ 
& $\myapprox \tfrac{2}{5}$ 
& $\myapprox \tfrac{4}{11}$ 
& $\myapprox \tfrac{1}{3}$ 
& $\myapprox \tfrac{2}{7}$ 
& $\myapprox \tfrac{4}{15}$ 
& $\myapprox \tfrac{1}{4}$ 
& $\myapprox \tfrac{2}{9}$ 
& $\myapprox \tfrac{4}{19}$ 
& $\myapprox \tfrac{1}{5}$ 
& $\myapprox \tfrac{2}{11}$ 
& $\myapprox \tfrac{4}{23}$ 
& $\myapprox \tfrac{1}{6}$ 
& $\myapprox \tfrac{4}{d+\lceil d/3 \rceil}$ 
\\[-1.2ex]
{\scriptsize ($G$ is bigon-free)} & & & & & & & & & & & & & & & & & \\
\hline
\end{tabular}
\renewcommand{\arraystretch}{1}
\caption{Bounds on the $d$-independence number of 1-planar
graphs. `$\myapprox$' means lower-order terms are omitted.  }
\label{ta:overviewNew}
\label{ta:resultsNew}
\end{table}

Our constructions for $d=3,4,5$ actually happen to have minimum degree $d$,
so requiring the minimum degree of the graph to be $d$ makes no difference
for $d=3,4,5$.    What happens for $d=6,7$?  (A 1-planar
graph has at most $4n-8$ edges, and so cannot have a minimum degree 8 or more.)   
For $d=6$ we can modify our construction slightly to achieve a minimum degree 6 while weakening the lower bound by only a constant
term. For $d=7$ we give an entirely different construction that has a fairly
large independent set, but we leave a multiplicative gap to the upper bound, see Table~\ref{ta:overview}.

\def\myapprox{\ensuremath{{\scriptstyle{\approx}}}}
\begin{table}[ht]
\hspace*{\fill}
\renewcommand{\arraystretch}{1.3}
\begin{tabular}{|r|c|c|c|c|c|}
\hline
$\delta(G)=$ & $3$ & $4$ & $5$ & $6$ & $7$ \\
\hline
\hline
$\alpha_{\delta(G)}(G)\leq$ & $\tfrac{6}{7}(n-2)$ & $\tfrac{2}{3}(n-2)$ & $\tfrac{4}{7}(n-2)$ & 
$\lceil \tfrac{1}{2}(n-4) \rceil$
& $\tfrac{4}{10}(n-2)$ \\
\hline
$\exists G: \alpha_{\delta(G)} \geq$ 
& $\tfrac{6}{7}(n-2)$ & $\tfrac{2}{3}(n-2)$ & $\tfrac{4}{7}(n-2)$ & $\tfrac{1}{2}(n-4)$ & $\tfrac{8}{21}(n-13.5)$ \\
\hline
\end{tabular}
\renewcommand{\arraystretch}{1}
\hspace*{\fill}
\caption{Bounds on the $d$-independence number of $1$-planar graphs with minimum degree 
$\delta(G)=d$.   All constructed graphs are simple.
} 
\label{ta:overview}
\end{table}

We also study the special situation concerns \emph{optimal 1-planar} graphs, which
have the maximum-possible number $4n-8$ of edges.   (This is motivated by the observation that the lower-bound constructions of Caro and Roditty all use planar
graphs with the maximum possible number of edges, whereas our earlier construction usually does \emph{not} have the maximum number of edges.)
Here we can show that the independence number is at most $\tfrac{2}{d}(n-2)$, and this is tight
if we allow parallel edges (and, for small $d$, even for simple optimal 1-planar graphs).
Our results are summarized in \Cref{ta:optimal}.

\begin{table}[ht]
\hspace*{\fill}
\renewcommand{\arraystretch}{1.3}
\begin{tabular}{|c|c|c|c|c|c|}
\hline

$d$ & 6 & 8 & 10 & 12 & $\geq 14$ even \\
\hline
\hline
$\alpha_{\delta(G)}(G)\leq$ & $\tfrac{1}{3}(n-2)$ & $\tfrac{1}{4}(n-2)$ & $\tfrac{1}{5}(n-2)$ &  $\tfrac{1}{6}(n-2)$ & $\tfrac{2}{d}(n-2)$ \\[-0.8ex]
\scriptsize ($G$ is bigon-free) & & & & & \\
\hline
$\exists G: \alpha_{\delta(G)} \geq$ 
& $\tfrac{1}{3}(n-2)$ & $\tfrac{1}{4}(n-2)$ & $\tfrac{1}{5}(n-2)$ & $\tfrac{1}{6}(n-4)$ & $\tfrac{2}{d}(n-2)$ \\[-0.8ex]
& \small (simple)
& \small (simple)
& \small (simple)
& \small (simple)
& \small (bigon-free) \\
\hline

\end{tabular}
\renewcommand{\arraystretch}{1}
\hspace*{\fill}
\caption{Bounds on the $d$-independence number of optimal $1$-planar graphs.} 
\label{ta:optimal}
\end{table}


Our paper is organized as follows.
After giving preliminaries, we first present 
in \Cref{upper_bounds} upper bounds for the $d$-independence number of $1$-planar graphs,
strongly inspired by techniques from \cite{BiedlW21} to bound matchings in $1$-planar graphs.
In \Cref{Examples}, we then construct a number of infinite families of $1$-planar graphs with large $d$-independent sets 
that match these bounds, and then further constructions when we require simplicity and/or minimum degree $d$.
\Cref{optimal} focuses on the independence number of optimal $1$-planar graphs.
Finally, we conclude our paper with further thoughts and pose some open questions.

\section{Preliminaries}

Let $G = (V, E)$ be a graph on $n$ vertices. 
We assume familiarity with basic terms in graph theory, such as connectivity
, and refer the reader to Bondy and Murty~\cite{Bondy} for graph theoretic notations. 
Throughout the paper, our input is always a connected graph $G=(V,E)$ on $n$ vertices, and $n \geq 3$. 
Graph $G$ may have \emph{parallel edges} (multiple edges that connect the same pair of vertices),
but it never has \emph{loops} (an edge that connects a vertex to itself).
We use the letter $I$ to denote an \emph{independent set} in $G$, i.e., a set of vertices in $G$ without edges between them.
The notation $\overline{I}$ refers to the set of vertices $V\setminus I$.

A drawing $\Gamma$ of a graph $G$ assigns vertices to points in $\mathbb R^2$ and edges to curves in $\mathbb R^2$ in such a way that edge-curves join the corresponding endpoints. 
In this paper we only consider 
drawings where the following holds
(see \cite{MR4336223} for extensive discussions on possible restrictions on drawings):
\begin{enumerate*}
    \item No vertex-points coincide and no edge-curve intersects a vertex-point except at its two ends.
    \item  If two edge-curves intersect at a point $p$ that is not a common endpoint, then they properly cross at $p$; the point $p$ is called a \emph{crossing}.
    \item If three or more edge-curves intersect in a point $p$, then $p$ is a common endpoint of
the curves. 
    \item If the curves of two edges $e, e'$ intersect twice at two distinct points $p,p'$, then $e, e'$ are parallel edges and $p, p'$ are their endpoints.
    \item If the curve of an edge $e$ self-intersects
    at point $p$, then $e$ is a loop and $p$ is its endpoint.
\end{enumerate*}

A drawing is called \emph{$k$-planar} if each edge has at most $k$ crossings; a $0$-planar drawing is simply called \emph{planar}. 
In this paper, all drawings are $1$-planar. 
A graph is called \emph{planar/1-planar} if it has a planar/1-planar drawing. 
A \emph{plane}/\emph{1-plane} graph is a graph together with a fixed planar/1-planar drawing.
For a given drawing $\Gamma$ of a graph, the \emph{cells} are the connected regions of 
$\mathbb R^2\setminus \Gamma$; if $\Gamma$ is planar, then the cells are also called \emph{faces}.   The unbounded cell is called the \emph{outer-face} (even for drawings that are not planar).
A \emph{bigon} is a cell $F$ that is bounded by two distinct parallel edges that have no crossings. 
We call a planar/1-planar graph \emph{bigon-free} if it has a plane/1-planar drawing that has no bigons.
It is easily derived from Euler's formula that a planar bipartite graph with $n\geq 3$ vertices
has at most $2n-4$ edges; this formula holds as long as (in some planar drawing) every face has at least three incident
edges, i.e., as long as the graph is bigon-free. 

\section{Upper bounds on the \texorpdfstring{$d$}{d}-independence number}\label{upper_bounds}

We use three different approaches to prove the upper bounds for the $d$-independence number in 1-planar graphs.
The first approach is to take a bound from \cite{BiedlW21} that was intended for 1-planar graphs with minimum degree $d$, and hence it is tight only for small values of $d$, namely, $d=3,4,5$.
For $d=6$, we use a simple edge-counting argument to improve this bound slightly.
For $d\geq 7$, neither of the previous approaches yields a satisfactory upper bound, but inspired by the proof in \cite{BiedlW21}, we 
use a charging/discharging argument as a third approach towards an upper bound.

\subsection{Upper bounds on the \texorpdfstring{$d$}{d}-independence number for \texorpdfstring{$d=3, 4, 5$}{d=3, 4, 5}}

In \cite{BiedlW21}, the first author and Wittnebel studied the sizes of maximum matchings in 1-planar graphs of minimum degree $\delta$ (for various values of $\delta$).   To obtain bounds on these, they needed the following lemma on 
independent sets in $1$-planar graphs.

\begin{lemma} {\rm\cite{BiedlW21}}
\label{lem:bipartite}
Let $G$ be a simple $1$-planar graph. 
Let $I$ be a non-empty independent set in $G$
where $\deg(t)\geq 3$ for all $t\in I$.
Let $I_d$ be the vertices in $I$ that have degree $d$. 
Then 
\begin{equation}
\label{eq:bipartite}
2|I_3|+\sum_{d\geq 4} (3d-6)|I_d| \leq 12|\overline{I}|-24.
\end{equation}
\end{lemma}

We use this lemma now for easy upper bounds on the size of a $d$-independent set $I$ in a 1-planar graph.
Write again $I_d$ for all vertices of $I$ that have a degree exactly $d$.

\begin{corollary}
In a simple $1$-planar graph, any  $3$-independent set $I$ satisfies   $|I|\leq \tfrac{6}{7}(n-2)$.
\end{corollary}
\begin{proof}
We have $2|I|=2\sum_{d\geq 3}|I_d| \leq 2|I_3|+\sum_{d\geq 4}(3d-6)|I_d|\leq 12(n-|I|)-24$, by \Cref{lem:bipartite}.
Therefore $14|I|\leq 12n-24$ which implies that $|I|\leq \tfrac{6}{7}(n-2)$.
\end{proof}
\begin{corollary}\label{cor:d=4}
In a simple $1$-planar graph, any  $4$-independent set $I$ satisfies   $|I|\leq \tfrac{2}{3}(n-2)$.
\end{corollary}
\begin{proof}
Since $I_3$ is empty, 
we have $6|I|=\sum_{d\geq 4}6|I_d| \leq 2|I_3|+\sum_{d\geq 4}(3d-6)|I_d|\leq 12(n-|I|)-24$ by \Cref{lem:bipartite}.
Therefore $18|I|\leq 12n-24$ which implies that $|I|\leq \tfrac{2}{3}(n-2)$.
\end{proof}
\begin{corollary}\label{cor:d=5}
In a simple $1$-planar graph, any  $5$-independent set $I$ satisfies   $|I|\leq \tfrac{4}{7}(n-2)$.
\end{corollary}
\begin{proof}
Since $I_3$ and $I_4$ are empty, 
we have $9|I|=\sum_{d\geq 5}9|I_d| \leq 2|I_3|+ \sum_{d\geq 4}(3d-6)|I_d|\leq 12(n-|I|)-24$
 by \Cref{lem:bipartite}.
 Therefore $21|I|\leq 12n-24$ which implies that $|I|\leq \tfrac{4}{7}(n-2)$.
\end{proof}

\subsection{Upper bounds for on the 6-independence number}

Note that if we apply \Cref{lem:bipartite} to a 6-independent set $I$, we can get a bound of $12|I|=\sum_{d\geq 6}12|I_d| \leq 2|I_3|+ \sum_{d\geq 4}(3d-6)|I_d|\leq 12(n-|I|)-24$
and therefore $24|I|\leq 12n-24$ which means $|I|\leq \tfrac{1}{2}(n-2)$. 
However, we are able to get a slightly better bound by using an alternative argument. 

\begin{lemma}
\label{lem:IS6}
Let $G$ be a simple $1$-planar graph. 
Then for any $d$-independent set $I$ 
we have $|I|\leq \tfrac{3n-8-\chi}{d}$
where $\chi=1$ if $n$ is odd and $\chi=0$ otherwise.
\end{lemma}
\begin{proof}
Consider the $1$-planar bipartite subgraph $G^-$
of $G$ that consists of only the edges between $I$ and $\overline{I}$.
This graph has $n$ vertices and has (by a result by Karpov \cite{Karpov14})
at most $3n-8-\chi$ edges.   
Since $I$ is a $d$-independent set, $d|I|\leq |E(G^-)|\leq 3n-8-\chi$
which implies the result.
\end{proof}

\begin{corollary}
In a simple $1$-planar graph, any  $6$-independent set $I$ satisfies   $|I|\leq \lceil \tfrac{1}{2}(n-4) \rceil$.
\end{corollary}
\begin{proof}
If $n$ is odd then $|I|\leq \tfrac{1}{6}(3n-9)=\tfrac{1}{2}(n-3)=\lceil \tfrac{1}{2}(n-4) \rceil$.
If $n$ is even  then 
$|I|\leq \tfrac{1}{6}(3n-8)= \tfrac{n}{2} - \tfrac{4}{3}$,
and by integrality of $|I|$ hence
$|I|\leq \tfrac{n}{2}-2 = \tfrac{1}{2}(n-4)$.
\end{proof}

Generally, the bound of \Cref{lem:IS6} will be the best upper bound that we can have found for simple graphs if $d\geq 6$ is divisible by 3.

\subsection{Upper bounds for \texorpdfstring{$d \geq 7$}{d ≥ 7}}

One can argue  that
applying \Cref{lem:bipartite,lem:IS6},
yields
 upper bounds for the $7$-independence number of $\approx \tfrac{3}{7}n$ and $\approx \tfrac{4}{9}n$, respectively.
We obtain an even better upper bound by using a charging/discharging argument.
This argument is based on the proof of \Cref{lem:bipartite} from \cite{BiedlW21}, but we modify it in two ways.   First,  as already hinted at in \cite{BiedlW21}, the bound can be improved when many vertices in $I$ have large degrees.   Second, the proof in \cite{BiedlW21} assumed simplicity, but inspection of the proof showed that it was used only for vertices of $I$ that have degree 3.
For vertices of higher degrees, we only need that  the graph is bigon-free (recall that this means that it has a 1-planar drawing without bigons) and can therefore write a more general statement.   Note that for $d=4,5$, the following lemma strengthens \Cref{cor:d=4,cor:d=5} in that it gives the same bound but does not require the graph to be simple.

\begin{lemma}
\label{lem:IS7}
Let $G$ be a bigon-free $1$-plane graph with a $d$-independent set $I$ for some integer $d\geq 3$.
Then $$|I|\leq \frac{4}{d+\lceil \tfrac{d}{3}\rceil}(n-2).$$
\end{lemma}
\begin{proof} 
We use a charging scheme, where we assign some {\em charges}
(units of weight) to edges in $G$ (as well as to some edges that we add 
to $G$), redistribute these charges to the vertices in $I$,
and then count the number of charges in two ways to obtain the 
bound.

As a first step, delete all edges of $G[\overline{I}]$ so that $G$ becomes bipartite.
Next, add edges to $G$ to make it maximal while not violating other assumptions.
So we add any edge to the fixed 1-planar drawing $\Gamma$ of $G$ that connects $I$ to $\overline{I}$, does not create a loop or a bigon, and can be added with at most one crossing.
Both operations can only increase the degrees 
of vertices in $I$, so it suffices to prove the bound in the resulting drawing $\Gamma'$.

As shown in \cite{BiedlW21}, for any vertex $t\in I$ there
cannot be three consecutive crossed edges in the circular ordering
of edges around $t$. For if there were three such edges (say $(t,s_1),(t,s_2),(t,s_3)$)
then the edge that crosses $(t,s_2)$ has one endpoint in $\overline{I}$; call it $x$.   We could
have added an uncrossed edge $(t,x)$ by walking along the part-edges towards the crossing in $(t,s_2)$.
In the circular ordering around $t$ this new edge $(t,x)$ would be before or after $(t,s_2)$, hence between two crossed edges, and so 
would not have formed
a bigon.   This contradicts maximality.

We assign charges as follows:  Let $E_-$ be the uncrossed edges of $\Gamma'$;
each of those receives 2 charges.  Let $E_\times$ be the crossed edges
of $\Gamma'$; each of those receives 1 charge.  
If we remove one edge from each pair of crossing edges, 
we retain $\tfrac{1}{2}|E_\times| 
+ |E_-|$ edges and the resulting graph is planar, bipartite, bigon-free, and so has at most $2n-4$ edges.
Hence, we have
\begin{eqnarray}
\label{eq:1}
\#\mbox{charges} & = & 2|E_-|+1|E_\times| 
\leq 4n-8
\end{eqnarray}

For $t\in I$, let $c(t)$ be the total charges of incident edges of $t$
and write $\deg(t)$ for the degree of $t$.    
We know that there
are at least $\lceil \tfrac{\deg(t)}{3} \rceil $ uncrossed edges at $t$ since
there are no three consecutive crossed edges.   Thus $t$ obtains
$2 \lceil \tfrac{\deg(t)}{3} \rceil $ charges from three uncrossed edges, and
at least $\deg(t)-\lceil \tfrac{\deg(t)}{3} \rceil$ further charges from the remaining edges.
Hence $c(t)\geq \deg(t)+\lceil \tfrac{\deg(t)}{3} \rceil \geq d+\lceil \tfrac{d}{3}\rceil$ and
\begin{eqnarray}
\label{eq:2}
\#\mbox{charges} & = & \sum_{t\in I} c(t) \geq |I|(d+\lceil \tfrac{d}{3}\rceil) 
\end{eqnarray}
Combining this with (\ref{eq:1}) gives $|I|(d+\lceil \tfrac{d}{3}\rceil) \leq 4n-8$ as desired.
\end{proof}


With this we have proved the upper-bound entries in Table~\ref{ta:resultsNew}.

\section{1-planar graphs with large \texorpdfstring{$d$}{d}-independent sets}
\label{Examples}
\label{sec:constructions}

In this section, we construct several families of 1-planar graphs
that have large $d$-independent sets (for various values of $d$).
We begin with bigon-free graphs, where we can match
the bounds of \Cref{lem:IS7} \emph{exactly} using only a few constructions and
techniques.   However, most of the resulting graphs are not simple; we then give
further constructions for $d=3$ and $d=6,\dots,18$ that yield simple 1-planar
graphs with large $d$-independent sets.
Many of our constructions follow a common approach, using nested cycles 
with vertices inserted between them that belong to the independent set $I$. 

\subsection{Constructions that match \texorpdfstring{\Cref{lem:IS7}}{Lemma 4.2}}


\paragraph{The case $d=3$:}
For $d=3$, the bound of \Cref{lem:IS7} evaluates to $n-2$, which is essentially meaningless, but as we argue now, can be matched.   This uses a subgraph $H_3$ that will be needed in later constructions as well.


\begin{lemma}
For any integer $N$, there exists a bigon-free 1-planar graph $G_3$ with $n\geq N$ vertices and a $3$-independent set of size $n-2$.
\end{lemma}
\begin{proof}
Consider the bigon-free graph $H_3$ shown in \Cref{fig:H3}(a) which is a 4-cycle $\{a,b,c,d\}$ with two non-incident edges duplicated.
Note that $\{a,c\}$ is a 3-independent set (in all our figures, vertices in the independent set $I$ are white while vertices of $\overline{I}$ are black).  
For $s\geq 1$, let $G_3^{(s)}$ be the graph obtained 
by taking $s$ copies of $H_3$, contracting all copies of $b$ into one, and contracting all copies of $d$ into one.
This gives $n=2s+2$ vertices (so by choosing $s$ sufficiently big, we have $n\geq N$), and the $s$ copies of $\{a,c\}$ give a 3-independent set of size $2s=n-2$.
\end{proof}

\begin{figure}[ht]
\centering
\subfloat[\centering ]{ {\includegraphics[scale=0.7,page=2,trim=0 0 0 0,clip]{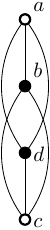}}
\quad
{\includegraphics[scale=0.7,page=3,trim=60 0 60 0,clip]{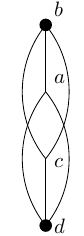}}
}
\qquad
\qquad
\subfloat[\centering]{\includegraphics[scale=0.7,page=4]{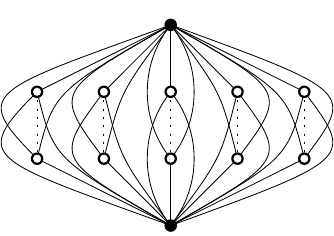}}
\caption{The graph $H_3$ (drawn in two different ways) and the graph $G_3$ with a 3-independent set of size $n-2$.   Dotted lines indicate the planar pairing.}
\label{fig:H3}
\label{fig:G3}
\end{figure}

\paragraph{The standard construction:}
We next construct graph families with large $d$-independent sets for $d\geq 4$.
Here we use a \emph{standard
construction} that we explain first in general terms  (see  Figure~\ref{fig:standard}).   
Fix three integer parameters $s,k,\tau$, where $s\geq 1$ is arbitrary (it serves to make the graph as big as we wish), while
$k\geq 3$ and $\tau$ will depend on parameter $d$.
We start with $s$ \emph{nested $k$-cycles}, i.e., 
cycles of length $k$ that are drawn (in the $1$-planar drawing that we construct)
such that each next cycle is inside the previous one.   We will show our
drawings on the \emph{standing flat cylinder}, i.e., a rectangle where the
left and right side have been identified; the nested cycles then become
horizontal lines.

The $s$ nested cycles define $s{+}1$ faces; of these, $s{-}1$ faces (the
\emph{middle faces}) are bounded by two disjoint nested cycles while two faces
(the \emph{end faces}) are bounded by one nested cycle.   Consider one middle
face, say it is bounded by nested cycles $P$ and $P'$.    We place $\tau$ vertices
$t_1,\dots,t_{\tau}$ inside this middle face;   these vertices (over all
middle faces together, plus possibly a few more at the end-faces) will be the independent set $I$. 
We make each $t_i$ adjacent to $\lceil d/2 \rceil$
vertices on one of $P,P'$ and $\lfloor d/2 \rfloor$ vertices on the other.
With this, the vertices in $I$ have degree $d$.    The main
bottleneck for $\tau$ and $k$ is that we must be able to place these vertices so that
the drawing is $1$-planar and bigon-free (and sometimes we also require the graph
to be simple or to have minimum degree $d$).
This will be mostly proved by pictures outlining the $1$-planar drawings.

\begin{figure}[ht]
\centering
 \subfloat[\centering ]{{\includegraphics[scale=0.2,height=5cm]{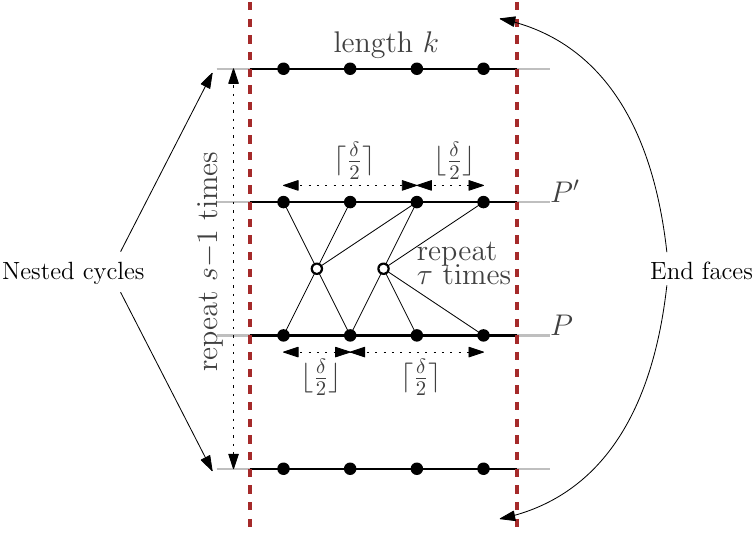} }}%
    \qquad
    \subfloat[\centering ]{{\includegraphics[scale=0.05,height=4cm]{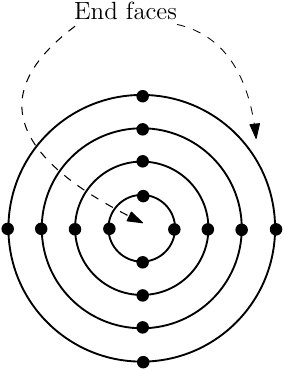} }}%
\caption{The standard construction 
on the standing cylinder and in the plane.}
\label{fig:standard}
\end{figure}

The construction inside the end-faces depends very much on $d$; sometimes
we add nothing at all, sometimes we add edges, sometimes we add more vertices
(in $I$ or in $\overline{I}$ or both). 
In total, the number of vertices is $n=s\cdot k + (s{-}1)\tau$ (plus whatever
we add in the end-faces). The size of the independent set is  $|I|=(s{-}1)\tau$
(plus whatever we add in the end-faces).

\paragraph{The case $d=4,5$:}
Using the standard construction, we now give graphs that have large $d$-independent
sets for $d=4,5$ and match \Cref{lem:IS7}.    As it turns out, we can do this and
even create simple graphs that have minimum degree $d$.

\begin{figure}[ht]
  \centering
    \subfloat[\centering ]{{\includegraphics[scale=0.7,page=3]{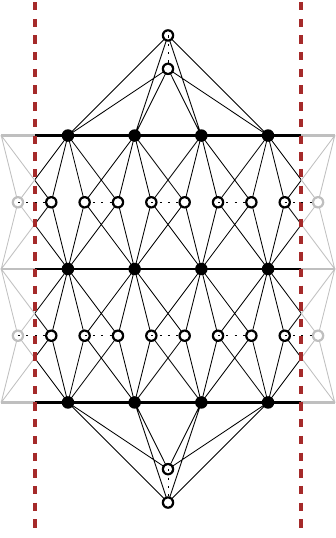} }}%
    \subfloat[\centering ]{{\includegraphics[scale=0.7,page=5]{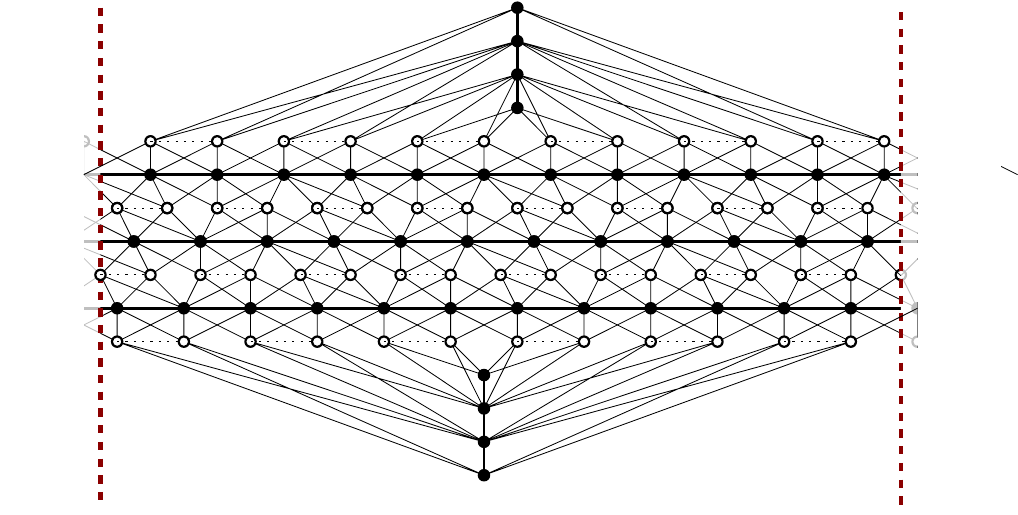} }}%
    \caption{The graphs $G_d$ for $d=4$ and $d=5$ (construction shown for $s=3$).   
    }
    \label{fig:G4}
    \label{fig:G5}
\end{figure}

\begin{lemma}
For any integer $N$ and $d\in \{4,5\}$, 
there exists a simple $1$-planar graph $G_d$ with minimum degree $d$ and $n\geq N$ vertices that has an
independent set of size $\tfrac{4}{d+\lceil d/3 \rceil}(n-2)$.
\end{lemma}
\begin{proof}
We follow the standard construction, choosing $s$ big enough so that the resulting graph has at least
$N$ vertices.   We choose $k$ and $\tau$ as follows:
\begin{itemize}
\item For $d=4$, we use $k=4$ (so nested quadrangles) and $\tau=8$.   Into each
	end-face, we add two vertices of $I$ that we make adjacent to all four
	vertices of the nested quadrangle that bounds the face.    
    See Figure~\ref{fig:G4}(b), and 
   verify that
   this can be done such that the drawing is 1-planar and the minimum degree is 4.
	We note that $n=4s+8(s{-}1)+4 = 12s-4$ and $|I|=8(s-1)+4=8s-4=\tfrac{2}{3}(12s-6)=\tfrac{2}{3}(n-2)$.

\item For $\delta=5$, we use $k=12$ and $\tau=16$.   Into each
	end-face, we add four vertices of $\overline{I}$ connected in a path, and then 12 vertices of $I$ that we each 
	make adjacent to two vertices of the path and three 
	vertices of the 12-gon that bounds the face.    See Figure~\ref{fig:G5}, and 
verify that
 this can be done such that the drawing is 1-planar and the minimum degree is 5.
	With this we have $n=12s+16(s{-}1)+32 = 28s+16$ and $|I|=16(s-1)+24=16s+8=\tfrac{4}{7}(28s+14)=\tfrac{4}{7}(n-2)$. \qedhere
\end{itemize}
\end{proof}
 
\paragraph{Planar pairings:}   For the cases $d\geq 6$, we use a convenient trick that allows us to increase
degrees in existing constructions.  Let $G$ be a 1-plane graph with an independent set $I$ of even size $2\ell$.
A \emph{planar pairing} of $I$ is an enumeration of the vertices in $I$
as $t_1,\dots,t_{2\ell}$ such that we could add the edges $(t_{2i-1},t_{2i})$ for $i=1,\dots,\ell$
to the fixed drawing of $G$ without adding any crossing.     One easily verifies that for all graphs that
we have constructed thus far, the independent sets have a planar pairing (indicated by dotted lines in
the respective pictures). 

If we have such a planar pairing, then we can (for $i=1,\dots,\ell$) insert a 1-plane subgraph $H$ at the place
for $(t_{2i-1},t_{2i})$ and identify two of its outerface vertices with $t_{2i-1}$ and $t_{2i}$.       
This increases the degrees of the vertices in $I$ and keeps a 1-planar drawing and the planar pairing of $I$.
If we are permitted to have parallel edges, then the best graph to use as $H$ is the graph $H_3$ from  Figure~\ref{fig:H3}(a).

\begin{theorem}
\label{thm:lb:bigonfree}
For $d\geq 3$, the $d$-independence number of bigon-free 1-planar graphs is exactly $\tfrac{4}{d+\lceil d/3 \rceil}(n-2)$.
\end{theorem}
\begin{proof}
The upper bound holds by \Cref{lem:IS7}.   For the lower bound,
we construct (by induction on $d$) graphs $G_d$
with $n_d$ vertices that have a $d$-independent set $I_d$ of size $\tfrac{4}{d+\lceil d/3 \rceil} (n_d-2)$;
furthermore, $I_d$ has a planar pairing.
We already did this for $d=3,4,5$ above. 
Now consider some $d\geq 6$ and set $d'=d-3$, $G'= G_{d'}$ and $I'=I_{d'}$.
At each pair $\{t,t'\}$ (of the planar pairing of $I'$), insert
a copy of $H_3$, see Figure~\ref{fig:G6}.

\begin{figure}[ht]
\centering
\subfloat[\centering]{\includegraphics[scale=0.9,page=5]{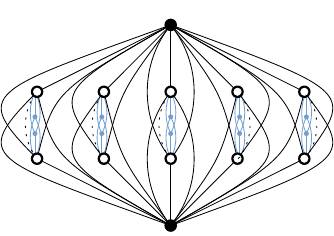}}
\qquad
\subfloat[\centering]{\includegraphics[scale=0.7,page=4]{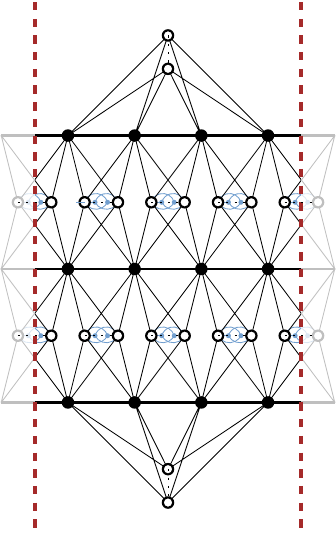}}
\caption{Graphs $G_6$ and $G_7$ are created by taking $G_3$ and $G_4$, respectively, and inserting $H_3$ (light blue) at each pair of the planar pairing.}
\label{fig:G6}
\label{fig:G7}
\end{figure}

In the resulting graph $G_d$, the independent
set $I_d:=I'$ now has degree $d$ or more at all vertices, and $G_d$ has $n=n'+|I'|$ vertices.   Therefore
\begin{align*}
\frac{n-2}{|I_d|} & = \frac{n'-2}{|I'|}+1 
= \frac{\delta'+\lceil \delta'/3 \rceil}{4} + 1
= \frac{(\delta'{+}3) +(\lceil \delta'/3 \rceil{+}1)}{4} 
= \frac{\delta+\lceil \delta/3 \rceil}{4}. \qedhere
\end{align*}
\end{proof}

\subsection{Simple graphs}

For $d=4,5$, the graphs that we constructed in the previous subsection were simple, but for all other
values of $d$ we used graph $H_3$ and therefore parallel edges.   In this subsection, we now construct
simple 1-planar graphs with $d$-independent sets that for $3\leq d\leq 7$ still match the known
upper bounds (up to small additive terms), but leave some gaps for $d\geq 8$.    We begin with
constructions for $d=3,6$.

\begin{lemma}
\label{lem:simple:IS3}
\label{lem:simple:IS6}
For any integer $N$ and $d\in \{3,6\}$ there exists a simple 1-planar graph $S_d$ with $n\geq N$ vertices such that the following holds:
\begin{enumerate}
    \item For $d=3$, $S_3$ has a $3$-independent set of size $\tfrac{6}{7}(n-2)$.
    \item For $d=6$, $S_6$ has a $6$-independent set of size $\tfrac{1}{2}(n-3)$.
\end{enumerate}
\end{lemma}
\begin{proof}
We use the standard-construction, with parameter $s$ large
enough so that the final graph has at least $N$ vertices.    The choice of parameters $k,\tau$ and the
modification at the end faces depends on $d\in \{3,6\}$
as follows:
\begin{itemize}
\item For $d=3$, we use $k=3$ (so nested triangles), and $\tau=18$.
	Into each end-face, we add three more vertices of $I$ that we make adjacent to all three
	vertices of the nested triangle that bounds the face.    See Figure~\ref{fig:S3}(a). 
	With this construction, we have $n=3s+18(s{-}1)+6 = 21s-12$ and $|I|=18(s-1)+6=18s-12=\tfrac{6}{7}(21s-14)=\tfrac{6}{7}(n-2)$.
\item For $d=6$ we use $k=\tau=3$ and do not add anything in the end-faces.  See Figure~\ref{fig:S6}(b).
	With this construction, we have $n=3s+3(s{-}1) = 6s-3$ and $|I|=3(s-1)=3s-3=\tfrac{1}{2}(6s-6)=\tfrac{1}{2}(n-3)$. \qedhere
\end{itemize}
\end{proof}

\begin{figure}[H]
    \centering
    \subfloat[\centering ]{{\includegraphics[scale=0.2,height=6cm]{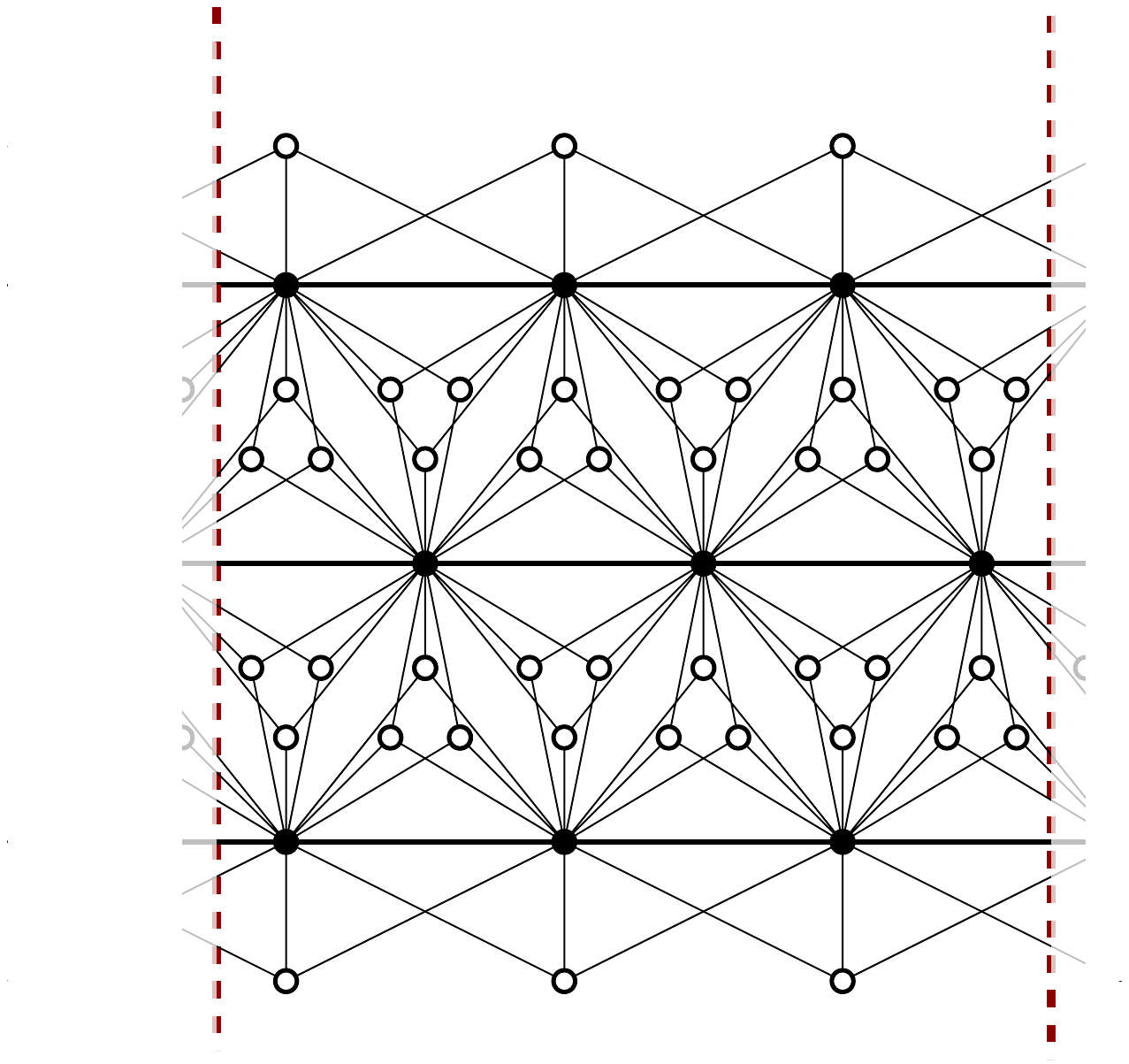} }}%
    \subfloat[\centering]{{\includegraphics[scale=0.6,trim=-20 0 20 0,clip,page=1]{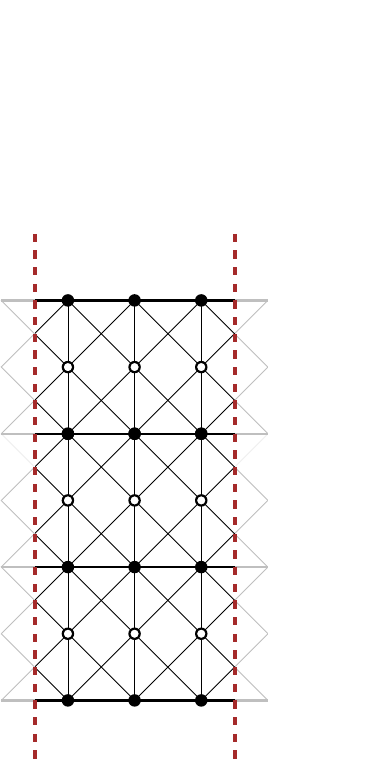} }}%
    \subfloat[\centering ]{{
	\includegraphics[scale=0.6,trim=0 0 0 0,clip,page=2]{delta=high.pdf} 
}}%
    \caption{Simple graphs for the $d$-independence number, for $d=3$ ($s=3$), $d=6$ ($s=4$) and $d=7$ ($s=5$).
    }
	\label{fig:S3}
	\label{fig:S6}
	\label{fig:S7}
\end{figure}

Next, we consider the case $d\in \{7,13,18\}$.  
For $d=13,18$, we need to construct 1-planar graphs that not only have a planar pairing among their independent set $I$, but that  have one more property
that we call a 
\emph{claw-cover}:   $|I|$ is divisible by 3, and there is  a set $C$ of $|I|/3$ vertices in $\overline{I}$  such that $C\cup I$ induces a set of $K_{1,3}$'s, i.e., each vertex in $C$ is incident to exactly three vertices of $I$, and each vertex in $I$ is incident to exactly one vertex of $C$.  

\begin{lemma}
\label{lem:mindeg1318:1sided}
For any integer $N$ and $d\in \{7,13,18\}$, there exists a simple 1-planar graph $S_d$ with $n\geq N$ vertices such that the following holds:
\begin{enumerate}
    \item For $d=7$, $S_7$ has a $7$-independent set of size $\tfrac{2}{5}(n-3)$ with a planar pairing.
    \item For $d=13$, $S_{13}$ has a $13$-independent set of size $\tfrac{1}{5}(n-6)$ with a planar pairing and a claw-cover.
    \item For $d=18$, $S_{18}$ has an $18$-independent set of size $\tfrac{1}{7}(n-6)$ with a planar pairing and a claw-cover.
\end{enumerate}
\end{lemma}
\begin{proof}
We use the standard construction with two modifications: we omit some of the edges of the nested cycles (to permit planar pairings), and we sometimes alternate the lengths of the nested cycles. As always we choose the number $s$ of nested cycles large
enough so that the final graph has at least $N$ vertices.    The length(s) of the nested cycles, the parameter $\tau$ and the
modification at the end faces depends on $d\in \{7,13,18\}$ as follows:
\begin{itemize}
\item For $d=7$, we use an odd number $s$ of nested cycles that alternate between length 3 and length 6 
(beginning and ending with length 3).   We also use $\tau=3$ and do not add anything in the end-faces.
See Figure~\ref{fig:S7}(c) for the construction and an illustration of the planar pairing.   
With this construction, we have $n=3\tfrac{s+1}{2}+6\tfrac{s-1}{2}+3(s-1)=\tfrac{15}{2}s-\tfrac{9}{2}$ and $|I|=3(s-1)=\tfrac{2}{5}(\tfrac{15}{2}s-\tfrac{15}{2})=\tfrac{2}{5}(n-3)$.
\item
For $d=13$, we use an odd number $s$ of nested cycles that alternate between length 6 and length 18
(beginning and ending with length 6). We also use $\tau=3$, and do not add anything in the end-faces. 
See Figure~\ref{fig:S13}(a) for the construction and an illustration of the planar pairing as well as the claw-cover.    
With this construction, we have $n=6\tfrac{s+1}{2}+18\tfrac{s-1}{2}+3(s-1)=15s-9$ and $|I|=3(s-1)=\tfrac{1}{5}(15s-15)=\tfrac{1}{5}(n-6)$.

\item
For $d=18$, we use $s$ nested 18-cycles, and use $\tau=3$. Into each end face, we add three more vertices of $I$ and 12 more vertices of $\overline{I}$.  
See Figure~\ref{fig:S18}(b) for the construction and an illustration of the planar pairing and the claw-cover.  
With this construction, we have $n=18s+3(s-1)+30=21s+27$ 
and $|I| =3(s-1)+6=3s+3 =\tfrac{1}{7}(21s+21)=\tfrac{1}{7}(n-6)$. \qedhere
\end{itemize}
\end{proof}

\begin{figure}[H]
    \centering 
    \subfloat[\centering]{{\includegraphics[scale=0.76,trim=70 0 60 40,clip,page=8,angle=90]{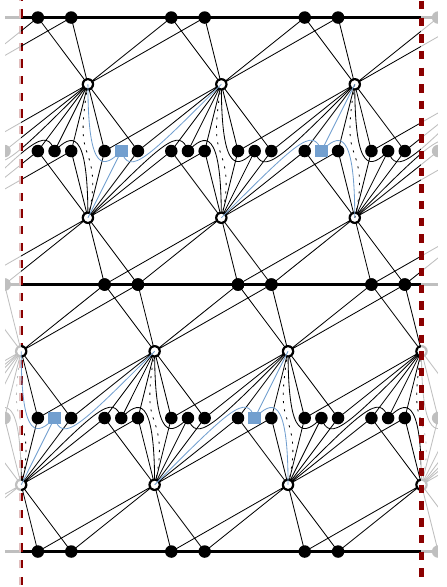} }}%
    \qquad \subfloat[\centering ]{{\includegraphics[scale=0.76,trim=40 30 30 40,clip,page=9,angle=90]{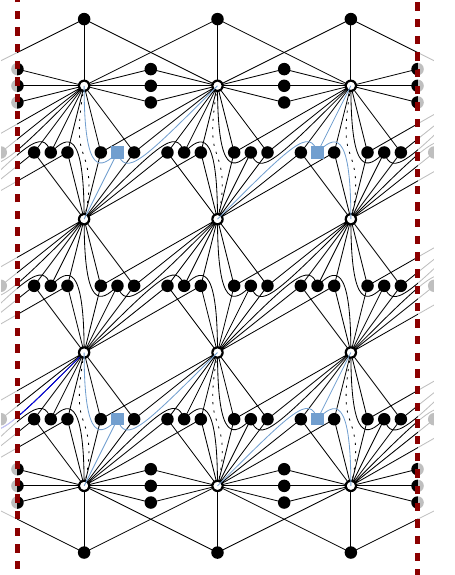} }}%
    \caption{Simple graphs for the $d$-independence number, for $d=13$ ($s=5$), and $d=18$ ($s=3$).
The claw-cover is indicated by light blue squares.
For ease of reading, we now show the construction on the rolling cylinder, rather than the standing one.
}
    \label{fig:S13}
    \label{fig:S18}
\end{figure}

\paragraph{Exploiting planar pairings and claw-covers:}
We have already seen in the proof of \Cref{thm:lb:bigonfree} how to use a planar
pairing to increase degrees by inserting a small subgraph.   We now do the
same here, but with a different (simple) subgraph.    We also can use claw-covers
to decrease degrees.   We first give two abstract results that explain exactly
how these operations affect the size of independent sets.

\begin{claim}
\label{cl:pairing}
Let $H$ be a simple 1-planar graph with $n'$ vertices that has (for some $A,B,d'$) a $d'$-independent set $I'$ of size $\tfrac{2}{d'-A}(n'-B)$ that has a planar pairing.
Then for any $d\geq d'$ there exists a simple 1-planar graph $G$ with $n\geq n'$ vertices and a  $d$-independent set $I$ of size $\tfrac{2}{d-A}(n-B)$.
\end{claim}
\begin{proof}
At any pair $\{t,t'\}$ of the planar pairing, insert $K_{2,x}$
for $x=d{-}d'$, i.e., 
add $x$ new vertices and make them adjacent to both $t$ and $t'$, see also Figure~\ref{fig:pair_claw}(a).
In the resulting graph $G$, all vertices of the independent
set $I:=I'$ have degree $d$ or more, and $G$ has $n=n'+\tfrac{1}{2}|I'|x$ vertices.   Therefore
\begin{align*}
& \frac{n-B}{|I|}= \frac{n'-B}{|I'|} + \frac{1}{2}x = \frac{d'-A}{2} + \frac{x}{2} = \frac{d'+x-A}{2} = \frac{d-A}{2}. \qedhere
\end{align*}
\end{proof}

\begin{claim}
\label{cl:claw}
Let $H$ be a simple 1-planar graph with $n'$ vertices and (for some $A,B,d'$) a $d'$-independent set $I'$ of size $\tfrac{3}{d'+A}(n'-B)$
that has a claw-cover.
Then for $d=d'{-}1$ there exists a simple 1-planar graph $G$ with $n= n'-|I|/3$ vertices and a $d$-independent set $I$ of size $\tfrac{3}{d+A}(n-B)$.
\end{claim}
\begin{proof}
Delete the $|I|/3$ vertices of $\overline{I}$ that belong to the claw-cover, See also Figure~\ref{fig:pair_claw}(b).
In the resulting graph $G$, the independent
set $I:= I'$ now has minimum degree $d$ and $G$ has $n=n'-\tfrac{1}{3}|I'|$ vertices.   Therefore
\begin{align*}
& \frac{n-B}{|I|}= \frac{n'-B}{|I'|} - \frac{1}{3} = \frac{d'+A}{3} - \frac{1}{3} = \frac{d'-1+A}{3} = \frac{d+A}{3}. \qedhere
\end{align*}
\end{proof}

\begin{figure}[H]
    \centering
    \subfloat[\centering]{{\includegraphics[scale=0.7,trim=30 0 30 40,clip,page=3,angle=90]{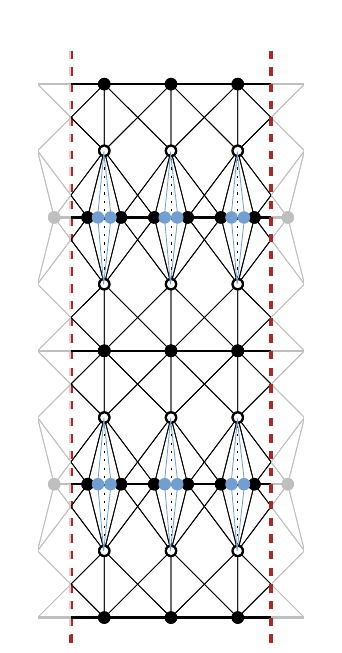} }}%
    \qquad
    \subfloat[\centering]{{\includegraphics[scale=0.7,trim=30 30 30 40,clip,page=10,angle=90]{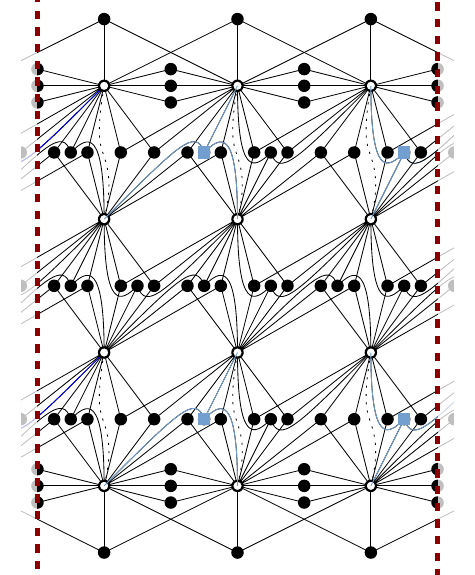} }}%
\caption{(a) Adding $K_{1,2}$ at the planar pairing of $S_7$ to get a simple 1-planar graph with a large 9-independent set.
(b) Removing the claw-cover from $S_{18}$ to get a simple 1-planar graph with a large 17-independent set
	(with the new claw-cover in blue squares).}	
    \label{fig:pair_claw}
\end{figure}

\paragraph{Filling Table~\ref{ta:resultsNew}:}
With this, we are ready to fill all entries 
of Table~\ref{ta:resultsNew}, i.e., lower bounds on $\alpha_d(G)$ for a simple 1-planar graph $G$.
We already had constructions that achieve the given bounds for $d\in \{3,4,5,6,7,13,18\}$.   The
remaining entries can be filled as follows:
\begin{itemize}
\item For $d=8,9,10,11$, use $S_7$ (which for $d'=7$ has a $d'$-independent set of size $\tfrac{2}{5}(n-3)=\tfrac{2}{d'-2}(n-3)$ with a planar pairing),
	and apply Claim~\ref{cl:pairing} to get a $d$-independent set of size $\tfrac{2}{d-2}(n-3)$.
\item For $d=12$, use $S_{13}$ (which for $d'=13$ has a $d'$-independent set of size $\tfrac{1}{5}(n-6)=\tfrac{3}{d'+2}(n-6)$ with a claw-cover)
	and apply Claim~\ref{cl:claw} to get a $d$-independent set of size $\tfrac{3}{d+2}(n-6)$.
\item For $d=14,15$, use $S_{13}$ (which for $d'=13$ has a $d'$-independent set of size $\tfrac{1}{5}(n-6)=\tfrac{2}{d'-3}(n-6)$ with a planar pairing)
	and apply Claim~\ref{cl:pairing} to get a $d$-independent set of size $\tfrac{2}{d-3}(n-6)$.
\item For $d=17$, use $S_{18}$ (which for $d'=18$ has a $d'$-independent set of size $\tfrac{1}{7}(n-6)=\tfrac{3}{d'+3}(n-6)$ with a claw cover),
	and apply Claim~\ref{cl:claw} to get a $d$-independent set of size $\tfrac{3}{d+3}(n-6)$.  One verifies that this independent set again has a claw-cover
	(see Figure~\ref{fig:pair_claw}(b)) so we can repeat the argument to get the same bound for $d=16$.
\item For $d\geq 19$, use $S_{18}$ (which for $d'=18$ has a $d'$-independent set of size $\tfrac{1}{7}(n-6)=\tfrac{2}{d'-4}(n-6)$ with a planar pairing),
	and apply Claim~\ref{cl:pairing} to get an independent set of size $\tfrac{2}{d-4}(n-6)$.
\end{itemize}

\subsection{Graphs with minimum degree \texorpdfstring{$d$}{d}} \label{sec:min-deg}

The simple 1-planar graphs $S_3$ (from Figure~\ref{fig:S3}) and
$G_4,G_5$ (from Figure~\ref{fig:G4}) 
not only had large $d$-independent sets for $d=3,4,5$,
but they had an even stronger property: They also had minimum degree $d$.    The constructed
graph $S_6$ (from Figure~\ref{fig:S6}) had only six vertices of degree 5, but the constructions 
for $d\geq 7$ had many vertices of degree less than $d$.    In this section, we briefly discuss what lower bounds on
the $d$-independence number we can achieve if we additionally require graphs to have minimum
degree $d$.    For $d=3,4,5$ the answer is obviously the same as before (i.e., $\tfrac{6}{7}(n-2)$,
$\tfrac{2}{3}(n-3)$ and $\tfrac{4}{7}(n-2)$).
But for $d=6,7$, we need to construct new graphs.   (The question is moot for $d\geq 8$ since
there are no 1-planar graphs with minimum degree $d\geq 8$ unless we have loops or bigons.)

For $d=6$, we only need a very small change.   Previously, we used the standard-construction
with $k=3$ (i.e., nested triangles) and inserted $\tau=3$ vertices of an independent set into
each middle face.   Let $M_6$ be the graph obtained if instead we
use $k=4$ (i.e., nested quadrangles), $\tau=4$, and insert a pair of crossing edges into each end-face, see Figure~\ref{fig:G6m}.    If there are $s$ nested quadrangles, then $M_6$ has
$n=4s+4(s-1)=8s-4$ vertices and an independent set of size $4(s-1)=4s-4=\tfrac{1}{2}(n-4)$. 

\begin{figure}[ht]%
\centering
\includegraphics[scale=0.6,page=4,angle=90]{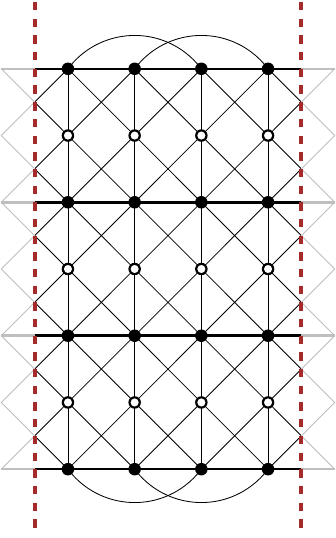}
     \caption{A graph $M_6$ with minimum degree 6 and an independent set of size $(n-4)/2$.  } 
    \label{fig:G6m}
\end{figure}

For minimum degree 7, we need to do significantly more work; in particular, it does
not seem possible to use the standard construction with nested cycles of constant length.
Instead, construct graphs of minimum degree 7 in two parts: first 
we give, by induction, a construction
with desirable properties and arbitrary size, and
then combine two such
constructions into a graph with a minimum degree 7.

\begin{claim}
For all $\ell\geq 0$, there exists a $1$-planar graph $M_7^{(\ell)}$ 
with $27\cdot 2^{\ell} - 9$
vertices and an independent set $I$ with $9\cdot 2^{\ell}-6$ vertices
such that (in some $1$-planar drawing)
\begin{itemize}
\itemsep -2pt
\item the outer-face is a cycle of $9\cdot 2^{\ell}$ vertices and these vertices have degree 4,
\item all other vertices have degree $7$,
\item no vertex of $I$ is on the outer-face.
\end{itemize}
\end{claim}
\begin{proof}
For the base case (${\ell}=0$), we need a graph with 18 vertices of which
three form an independent set; see Figure~\ref{fig:M7}(a) for such a graph that 
satisfies
all conditions.

Now assume that we have graph $M_7^{({\ell})}$ with $27\cdot 2^{\ell}-9$ vertices and $9\cdot 2^{\ell}$ vertices on the
outer-face $F_{\ell}$.  
We show how to construct $M_7^{({\ell}+1)}$ meeting all conditions.
Insert $9\cdot 2^{\ell}$ new vertices
in $F_{\ell}$ (let $I_{{\ell}+1}$ be the set of added vertices) and make each of them adjacent
to three vertices of $F_{\ell}$; \Cref{fig:M7}(b) shows that this
can be done while retaining $1$-planarity and keeping $I_{{\ell}+1}$
on the outer-face.   
With this, all vertices
in $F_{\ell}$ receive three more neighbours and hence now have degree~7.
Insert $18\cdot 2^{\ell}$ new vertices into the outer-face of the resulting
graph, and connect them in a cycle that will form the outer-face  $F_{{\ell}+1}$
of $M_7^{({\ell}+1)}$. 
Make each vertex of 
$I_{{\ell}+1}$ adjacent to four vertices of $F_{{\ell}+1}$; the figure shows that
this
can be done while remaining $1$-planar.   
Also, with this all vertices
on $F_{{\ell}+1}$ receive two neighbours in $I_{{\ell}+1}$; this plus the cycle
among them ensures that they have degree~4 while all other vertices
have degree~7.   
As desired $I_{{\ell}+1}$ forms an independent set
and has no edges to vertices of the independent set $I_\ell$
of $M_7^{({\ell})}$ 
since those are not on $F_{\ell}$ by the inductive hypothesis.

It remains to verify the claim on the size.   
Independent set $I_{\ell}\cup I_{{\ell}+1}$ has size
$9\cdot 2^{\ell}-6+9\cdot 2^{\ell} = 9\cdot 2^{{\ell}+1}-6$.    The outer-face
$F_{{\ell}+1}$ of $M_7^{({\ell}+1)}$ has $18\cdot 2^{\ell}=9\cdot 2^{{\ell}+1}$ vertices,
and finally $|V(M_7^{({\ell}+1)})|=|V(M_7^{({\ell})})|+|I_{{\ell}+1}|+|F_{{\ell}+1}|
=  27\cdot 2^{\ell}-9 + 9\cdot 2^{\ell} + 18 \cdot 2^{\ell} = 27\cdot 2^{{\ell}+1}-9$.
\end{proof}

    \begin{figure}[ht]%
    \centering
\subfloat[\centering ]{\includegraphics[angle=-90,scale=0.7,page=4,trim=0 40 0 45,clip]{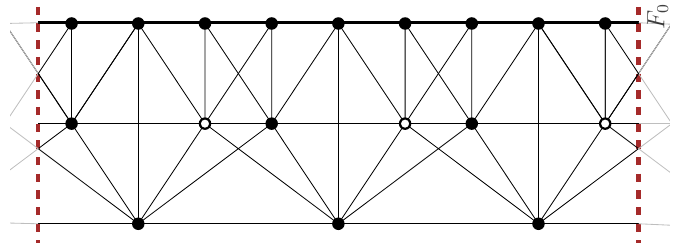}}
    \qquad
\subfloat[\centering ]{
    \includegraphics[angle=-90,scale=0.7,page=5,trim=0 00 0 45,clip]{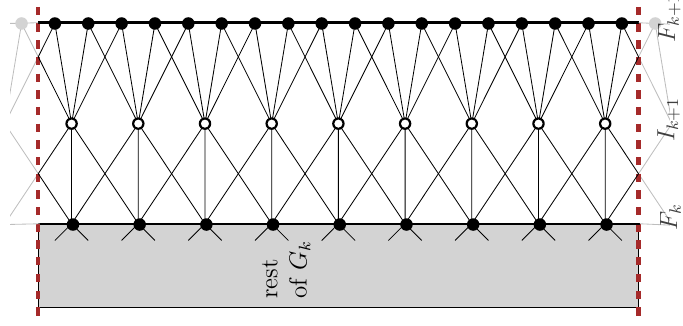}}%
        \qquad
\subfloat[\centering ]{
    \includegraphics[angle=-90,scale=0.7,page=6,trim=0 00 0 0,clip]{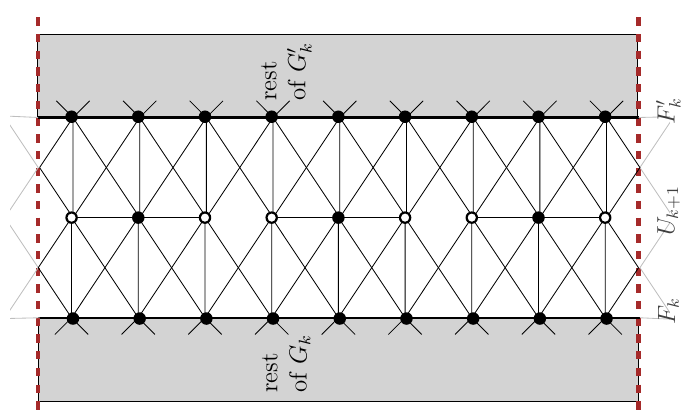}}%
    \caption{The base case and the induction step for building
the graph $M_7^{(k)}$, and combining two copies of $M_7^{(k)}$.   
}
    \label{fig:M7k}
    \label{fig:M7}
\end{figure}

\begin{lemma}
\label{lem:mindeg7}
For any integer $N$,
there exists a simple $1$-planar graph with minimum degree $7$ and $n\geq N$ vertices with an independent set
of size $\tfrac{8}{21}(n-13.5)$.
\end{lemma}
\begin{proof}
Let $k=\lceil \log_2((N{+}18)/63)\rceil $ and start with two copies of $M_7^{(k)}$, placed
such that the two outer-faces $F_k,F_k'$ of the two copies together bound one face. 
Into this face, insert $9\cdot 2^k$ vertices that we call $U_{k+1}$, 
grouped into $3\cdot 2^k$ paths of three vertices each.  Make  each
vertex of $U_{k+1}$ adjacent
to three vertices each of $F_k$ and $F_k'$;
\Cref{fig:M7k}(c) shows that this can be done while remaining $1$-planar.   

Each vertex of $U_{k+1}$ has at least one neighbour  in $U_{k+1}$,
and three neighbours in $F_k$ and three neighbours in $F_k'$.
Thus the resulting graph $G$ has minimum degree 7.   Define $I$ to consist of the
two independent sets of the two copies of $M_7^{(k)}$ 
as well as the $6\cdot 2^k$ end-vertices of the paths in $U_{k+1}$; this
is an independent set, see Figure~\ref{fig:M7k}(c).

It remains to analyze the size of $G$ and $I$.   Since $G$ contains
two copies of $M_7^{(k)}$, plus $U_{k+1}$, it has
$$n = 2(27 \cdot 2^k - \cdot 9) + 9\cdot 2^k = 63\cdot 2^k-18$$
vertices,   
hence $n\geq N$ by our choice of $k$.
Likewise, $I$ contains two copies of the independent set of
$M_7^{(k)}$, plus the ends of the $3\cdot 2^k$ paths, hence
$$|I|=2(9\cdot 2^k - \cdot 6) + 6\cdot 2^k = 24\cdot 2^k-12.$$
Since $\tfrac{8}{21}(n-13.5)=\tfrac{8}{21}(63\cdot 2^k-18-13.5)=24\cdot 2^k -\tfrac{144}{21}-\tfrac{108}{21}=24\cdot 2^k-12=|I|$, the bound holds.
\end{proof}

With this we have proved the lower-bound entries in~\Cref{ta:overview} for $d=3,\dots,7$.

\section{Optimal 1-planar graphs}\label{optimal}
\label{sec:optimal}

Recall that one of the motivations for our study was the paper by 
Caro and Roditty~\cite{MR0824858}, who showed that a planar graph 
$G$ 
with minimum degree $\delta$ has $\alpha(G)\leq \frac{2n-4}{\delta}$.
We repeat the proof here to observe that 
`minimum degree' is not required (as long as we bound the degree of the vertices in the independent set),
and that the bound also holds for graphs that are
not necessarily simple (as long as 
there are no bigons).

\begin{lemma} (based on {\rm\cite{MR0824858}}) \label{lem:ub:planar}
Let $G$ be a planar bigon-free graph.
Then for any $d$-independent set $I$ (for $d\geq 1$) we have $|I|\leq \tfrac{2}{d}(n-2)$.
\end{lemma}
\begin{proof}%
Let $I$ be any $d$-independent set and let $\overline{I}:=V\setminus I$.
Delete any edges between vertices of $\overline{I}$ to obtain a planar
bipartite graph $G'$.  Since $G$ has no bigons, graph $G'$ has no bigons either and therefore it has at most $2n-4$ edges.
Since every vertex in $I$
is incident to at least $d$ of these edges, we have $d|I|\leq 2n-4$
or $|I|\leq \tfrac{2}{d}(n-2)$.
\end{proof}

Caro and Roditty showed this bound is tight for $d\leq 5$ by constructing planar graphs $G$ with minimum degree $d$ and $\alpha(G)=\frac{2n-4}{\delta}$.
(We will show below that Lemma~\ref{lem:ub:planar}
is actually tight for all $d$.)
Inspection of their construction shows that these are
\emph{maximal planar graphs}, i.e., planar graphs that have  the maximum possible number $3n-6$ of edges.

In the same spirit, we ask what the independence number can be for 1-planar graphs that have the maximum possible number of edges.   It is known that every $1$-planar graph has at most $4n-8$ edges, and a bigon-free $1$-planar graph $G$ is called \emph{optimal} if it has exactly $4n-8$ edges.   An optimal $1$-planar graph can equivalently be defined as the graph obtained by taking
a \emph{quadrangulated graph} $Q$ (i.e., a graph with a planar drawing where all faces are bounded by 4-cycles)
and inserting a pair of crossing edges into each face.   Numerous results
are known for optimal 1-planar graphs, see \cite{BSW84}. In particular, 
an optimal 1-planar graph has
$n-2$ pairs of crossing edges.

\begin{lemma}
\label{lem:ub:optimal}
Let $G$ be an optimal $1$-planar graph.
Then for any $d$-independent set $I$ (for $d\geq 3$) we have $|I|\leq \tfrac{2}{d}(n-2)$.
\end{lemma}
\begin{proof}
Fix an arbitrary vertex $t\in I$; we know $\deg(t)\geq d$.   In any
$1$-planar drawing of $G$, the cyclic order of edges around $t$ alternates
between uncrossed and crossed edges, see \cite{BSW84}.  
Therefore half of
the incident edges of $t$ are crossed, and we assign all these crossings to $t$.
This does not double-count crossings, because (in an optimal $1$-planar
graph) the four endpoints of a crossing induce $K_4$ and so at most one
of them can belong to $I$.    We assigned at least $d/2$ crossings to
every vertex in $I$, and there are exactly $n{-}2$ crossings, so $|I|\leq \tfrac{2}{d}(n-2)$.
\end{proof}

We now construct bigon-free optimal $1$-planar graphs for which this bound is tight.
Note that in an optimal $1$-planar graph all vertex-degrees are even, so we will
only consider even values of $d$.   We first need 
to construct some quadrangulated graphs to 
show that the bound of \Cref{lem:ub:planar}
is tight. 

\begin{theorem}
\label{thm:planar}
For $d\geq 2$, the $d$-independence number of simple planar quadrangulated graphs is exactly $\tfrac{2}{d}(n-2)$.
\end{theorem}
\begin{proof}
The upper bound holds by \Cref{lem:ub:planar}.   For the lower bound, consider first the graph $K_{2,n}$, 
which is simple planar quadrangulated and has a 2-independent set $I'$ of size $n-2=\tfrac{2}{d'}(n-2)$
for $d'=2$.     Furthermore, if we choose $n$ even then $I'$ has a planar pairing.   Applying Claim~\ref{cl:pairing}
with $A=B=0$ we hence can obtain a graph that has a $d$-independent set of size $\tfrac{2}{d}(n-2)$ for all $d\geq 2$,
and inspection of the construction (which inserts $K_{2,x}$
for $x=d-2$ in place of each edge of the planar pairing)
one easily verifies that the resulting graph is in fact simple, planar and quadrangulated.
\end{proof}

\begin{theorem}
\label{thm:optimal}
For $d\geq 2$, the $2d$-independence number of bigon-free optimal 1-planar graphs is exactly $\tfrac{2}{2d}(n-2)$.
\end{theorem}
\begin{proof}
The upper bound holds by \Cref{lem:ub:optimal}.   For the lower bound, we know from Theorem~\ref{thm:planar} that
there exists a simple planar quadrangulated graph $Q_d$ with an independent set $I_d$ of size $\tfrac{2}{d}(|V(Q_d)|-2)$.
Now obtain graph $O_d$ by adding the dual graph $Q_d^*$ to $Q_d$ and connecting every dual vertex $v_F$ of $Q_d^*$ 
to all vertices of the face $F$ of $Q_d$ that $v_F$ represents.   
It is well-known that this gives a bigon-free optimal $1$-planar graph, see also Figure~\ref{fig:optimal}.

Since $Q_d$ is quadrangulated, it has exactly $|V(Q_d)|-2$ faces.    Therefore $O_d$ has $n=2|V(Q_d)|-2$ vertices
and the independent set $I_d$ of $Q_d$ is also an independent set of $O_d$ and has size
$|I_d|=\tfrac{1}{d}(2|V(Q_d)|-4)=\tfrac{1}{d}(n-2)$ as desired.
\end{proof}

\begin{figure}[ht]
\centering
\subfloat[\centering ]{ \includegraphics[scale=0.7,page=2,trim=30 0 50 0,clip]{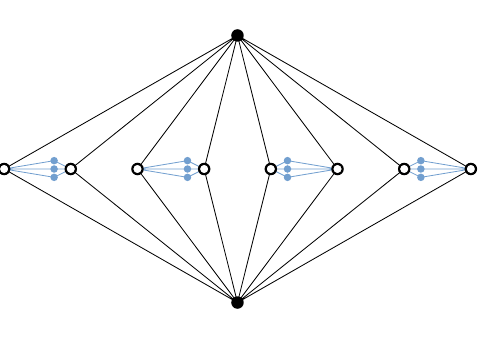}}%
    \qquad
\subfloat[\centering ]{ \includegraphics[scale=0.7,page=3,trim=0 00 0 0,clip]{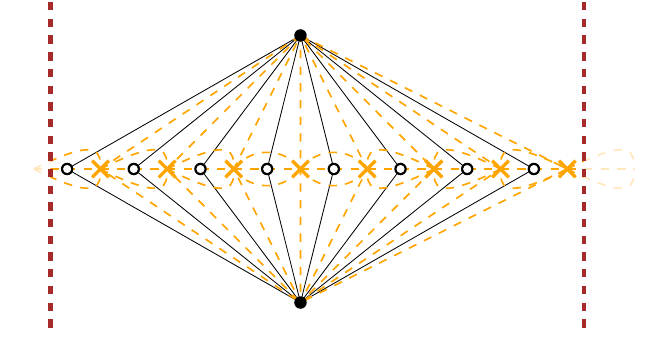}}%
\caption{(a) The graph $K_{2,n}=:Q_2$ (black); after inserting $K_{2,d-2}$ (here $d=5$) at a planar pairing it has a $d$-independent set of size $\tfrac{2}{d}(n-2)$.   (b) The optimal 1-planar graph $O_2$ corresponding to $Q_2$; it has a 4-independent set of size $\tfrac{1}{2}(n-2)$.}
\label{fig:optimal}
\end{figure}

The optimal 1-planar graphs $O_d$ were constructed via the planar quadrangulated graphs $Q_d$, which have vertices of
degree 2.   At each such vertex, the dual graph has two parallel edges, which means that $O_d$ is not simple for
any $d\geq 2$.   We briefly sketch here that for $3\leq d\leq 5$ we can actually  construct simple optimal
1-planar graphs that achieve the bound on the $d$-independence number.   To this end, it suffices to construct
planar quadrangulated graphs $\hat{Q}_d$ that are 3-connected.    We can do this with the standard construction, after omitting the edges within the nested cycles, 
as follows:
\begin{itemize}
\item For $d=3$, we use $k=4, \tau=8$, and insert two vertices of the independent set into each end face.
\item For $d=4$, we use $k=4, \tau=4$, and insert one vertex of the independent set into each end face.
\item For $d=5$, we alternate the length of the nested cycles between 5 and 10 (beginning and ending with 5).
	We use $\tau=4$, and insert one vertex of the independent set into each end face.
\end{itemize}
See also Figure~\ref{fig:planar:3conn}.   One easily verifies that the independent sets of these graphs 
has size $\tfrac{2}{d}(n-2)$ and the graphs are planar, quadrangulated and 3-connected.   Combining these
graphs with their duals as in the proof of \Cref{thm:optimal} therefore gives simple optimal 1-planar graphs
with a $(2d)$-independent set of size $\tfrac{1}{d}(n-2)$ for $d\in \{3,4,5\}$.

\begin{figure}[ht]
\centering
\subfloat[\centering ]{{\includegraphics[page=1,scale=0.7]{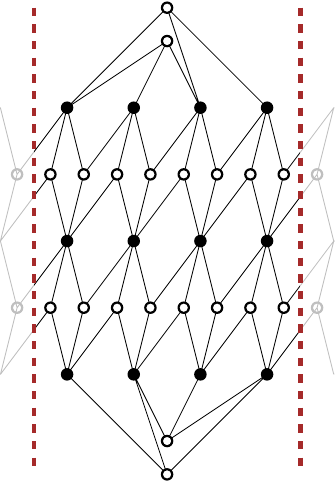}}}
\qquad
\subfloat[\centering ]{{\includegraphics[page=2,scale=0.7]{planar.pdf}}}
\qquad
\subfloat[\centering ]{{\includegraphics[page=3,scale=0.7]{planar.pdf}}}
\caption{Planar quadrangulated graphs with a $d$-independent set of size $\tfrac{2}{d}(n-2)$, for $d=3,4,5$.}
\label{fig:planar:3conn}
\end{figure}

\section{Open problems}

In this paper, we studied the $d$-independence number of 1-planar graphs, both in the setting
where the entire graph has to have minimum degree $d$ and in the setting where only
the vertices of the independent set have to have degree $d$ or more.   We provide
both lower and upper bounds, which for small values of $d$ are matching.   We leave
a few open problems:

\begin{itemize}
\item For the $d$-independence number for simple graphs, we leave significant gaps between
lower and upper bounds for $d\geq 8$.   We suspect that in particular the upper bound
(i.e., Lemma~\ref{lem:IS7}) could be improved further for $d\geq 8$ by exploiting that the 
graph is simple.  Doing such an improvement would give an interesting insight into the structure of the neighbourhood
of vertices of high degree in simpler 1-planar graphs.
\item In a similar spirit, can we strengthen the bound of Lemma~\ref{lem:IS7}
if we know that \emph{all} vertices (not just those in the independent set $I$)
have a degree of at least 7?
\item What is the $d$-independence number of other classes of near-planar graphs?    One
could ask this question both for generalizations of 1-planar graphs (such as 2-planar graphs or
fan-planar graphs) as well as subclasses of 1-planar graphs such as IC-planar graphs.
\item Last but not least, it would be interesting to explore algorithmic questions around finding independent sets of a certain size.   For example, it is easy to find an independent set of size $\tfrac{n}{8}$ in any 1-planar graph (because they are 7-degenerate and so can be 8-coloured in linear time).   With more effort, we can even 7-color the graph in linear time, so find an independent set of size at least $\tfrac{n}{7}$ \cite{CK05}.    
But can we find, say, a $6$-independent set of size $\tfrac{n}{3}-O(1)$ in an optimal 1-planar graph efficiently?   
\end{itemize}

\bibliographystyle{plain}
\bibliography{refs}

\end{document}